\documentclass[12pt, a4, reqno]{amsart}
\usepackage[a4paper, left=28mm, right=28mm, top=28mm, bottom=34mm]{geometry}
\usepackage[all]{xy}
\usepackage{amsmath}
\usepackage{amscd}
\usepackage{amssymb}
\usepackage{amsthm}
\usepackage{comment}

\usepackage{mathrsfs}
\theoremstyle{definition}
\newtheorem{thm}{Theorem}[section]
\newtheorem{lem}[thm]{Lemma}
\newtheorem{cor}[thm]{Corollary}
\newtheorem{prop}[thm]{Proposition}
\newtheorem{conj}[thm]{Conjecture}
\theoremstyle{definition}

\newtheorem{rem}[thm]{Remark}

\newtheorem{definition}[thm]{Definition}

\begin{document}

\title[Chow motives of quasi elliptic surfaces]{Chow motives of quasi elliptic surfaces}

\author{Daiki Kawabe}
\address{Tohoku University, Aoba, Sendai, 980-8578, Japan}
\email{daiki.kawabe.math@gmail.com}

\date{October 24, 2023}
\maketitle

\begin{abstract}
We prove that the transcendental motive of any quasi elliptic surface is trivial.
To prove this, we focus on the uniruledness of quasi elliptic surfaces.
\end{abstract}

\section{Introduction}
Let $k$ be an algebraically closed field of characteristic $p \geq 0$.
Let $X$ be a smooth projective surface over $k$ and $h(X)$ its Chow motive with $\mathbb{Q}$-coefficients.
Kahn-Murre-Pedrini \cite{KMP} proved that $X$ admits a refined Chow-K$\ddot{\mathrm{u}}$nneth decompositon 
$h(X) \cong \oplus_{i=0}^{4}h_{i}(X)$ with 
\[ h_{2}(X) \cong h_{2}^{alg}(X) \oplus t_{2}(X). \]
The motive $t_{2}(X)$ is called the transcendental motive of $X$.
It is a birational invariant and, for a prime number $l \neq p$, 
\[ H^{*}_{\textit{\'et}}(t_{2}(X)) = H^{2}_{\textit{\'et}}(X, \mathbb{Q}_{l})_{tr} \ \ \ \text{and} \ \ \ \mathrm{CH}^{*}(t_{2}(X)) = T(X)_{\mathbb{Q}}, \]
where $H^{2}_{\textit{\'et}}(X, \mathbb{Q}_{l})_{tr}$ is the transcendental lattice and $T(X)_{\mathbb{Q}}$ is the Albanese kernel.
The motives $h_{i}(X)$ (for $i \neq 2$) and $h_{2}^{alg}(X)$ are well understood, 
but the transcendental motive $t_{2}(X)$ is still mysterious.
For example, there is the following conjecture$:$
\begin{conj} (Conservativity) \label{Bloch conj}
If $H^{*}_{\textit{\'et}}(t_{2}(X)) = 0$, then $t_{2}(X) = 0$.
\end{conj}
When $k = \mathbb{C}$,
Conjecture \ref{Bloch conj} is equivalent to the famous conjecture of Bloch \cite{Bloch}.
It is known for surfaces over $\mathbb{C}$ of Kodaira dimension $\kappa < 2$,
but is wide open for surfaces of $\kappa = 2$
(e.g. \cite{PW} for some examples of surfaces where Conjecture \ref{Bloch conj} is proved).\\
\indent In this paper, we prove Conjecture \ref{Bloch conj} for quasi-elliptic surfaces, 
which can exist in characteristic $2$ and $3$, only.
More precisely, the purpose of this paper is to prove the following$:$
\begin{thm} \label{Intro theorem}
Let $f : X \rightarrow C$ be a quasi-elliptic surface.
Then  
\[ t_{2}(X) = 0. \]
\end{thm}
\subsection{Organization}
This paper is organized as follows.
In Section $2$, we recall the definitions and properties of uniruled surfaces, Shioda-supersingular surfaces, and quasi-elliptic surfaces.
In this paper, we focus on the uniruledness of quasi-elliptic surfaces (Theorem \ref{coverling for quasi-elliptic}).
In Section $3$, we prove two lemmas about homomorphisms between transcendental motives
(Lemma \ref{equivalence for t_{2}} and Lemma \ref{finite morphism for t_{2}}). 
In Section $4$, we prove Theorem \ref{Intro theorem}.
More precisely, we prove that the transcendental motive of any uniruled surface is zero (Theorem \ref{Main 1}).
\subsection{Notation}
Throughout this paper, let $k$ be an algebraically closed field of characteristic $p \geq 0$ and 
let $\mathcal{V}(k)$ be the category of smooth projective varieties over $k$.

\section{The uniruledness of quasi elliptic surfaces}

\subsection{Uniruled surfaces} \ \indent \\
In this subsection, we recall the notions of uniruledness and birationally ruledness.\\
Let $X \in \mathcal{V}(k)$ be a surface.
\begin{enumerate}
\item We say $X$ is \textit{uniruled} if there exist a curve $C$ and a dominant rational map
\[ \phi \colon \mathbb{P}^1 \times C \dashrightarrow X. \]
We say $X$ is \textit{separably uniruled} (resp. \textit{purely inseparable uniruled}) if there exists a such a rational map $\phi$ inducing 
a \textit{separable} (resp. \textit{purely inseparable}) extension of function fields.
\item We say $X$ is \textit{birationally ruled} if there exist a curve $C$ and a \textit{birational map}
\[ \phi \colon \mathbb{P}^1 \times C \overset{\cong}\dashrightarrow X. \]
\end{enumerate}

\indent 
The following fact is well-known.

\begin{prop} \label{separably uniruled and birationally ruled} 
Let $X \in \mathcal{V}(k)$ be a surface.
Then the following are equivalent:
\begin{enumerate}
\item $X$ is birationally ruled$;$
\item $X$ is separably uniruled$;$
\item $X$ has negative Kodaira dimension.
\end{enumerate}
\end{prop}
\begin{proof}
(i) $\Rightarrow$ (ii)$:$ This is clear.
(ii) $\Rightarrow$ (iii)$:$
If $X$ is separably uniruled, then there exist a curve $C$ and a dominant rational map 
$\phi : \mathbb{P}^{1} \times C \dashrightarrow X$ such that
the extension of function fields $k(\mathbb{P}^{1} \times C)/k(X)$ is separable.
Then $P_{n}(\mathbb{P}^{1} \times C) = P_{n}(\mathbb{P}^{1}) \cdot P_{n}(C) = 0 \cdot P_{n}(C) = 0$
for every $n \geq 1$.
Since $k(\mathbb{P}^{1} \times C)/k(X)$ is separable, $P_{n}(\mathbb{P}^{1} \times C) \geq P_{n}(X)$.
Thus $P_{n}(X) = 0$ for every $n \geq 1$.
Namely, $X$ has negative Kodaira dimension.\\
\indent (iii) $\Rightarrow$ (i)$:$ For example, see \cite[Theorem 13.2, p.195]{Badescu}
\end{proof}

To derive the uniruledness of quasi-elliptic surfaces, we need the following$:$
\begin{thm} (Noether-Tsen) \label{Noether-Tsen}
Let $\phi : Y \dashrightarrow B$ be a dominant rational map from a surface $Y$ to a curve $B$ satisfying the following conditions $:$
\begin{enumerate}
\item $k(B)$ is algebraically closed in $k(Y);$
\item The generic fiber of $\phi$ has arithmetic genus $0$.
\end{enumerate}
\indent Then $Y$ is birationally-isomorphic to $\mathbb{P}^1 \times B$.

\end{thm}
\begin{proof}
For example, see \cite[Theorem\ 11.3,\ p.166]{Badescu}.
\end{proof}

\subsection{Shioda-supersingular surfaces} \ \indent \\
In this subsection, we recall the notions of Lefschetz numbers and Shioda-supersingularity.\\
Let $X \in \mathcal{V}(k)$ be a surface.
Let $\mathrm{Br}(X) : = H^{2}_{\textit{\'et}}(X, \mathbb{G}_{m})$
denote the cohomological Brauer group of $X$.
For a prime number $l \neq p$, we consider the $l$-adic Tate module
\[ T_{l}(\mathrm{Br}(X)) : = \lim_{\longleftarrow n} \mathrm{Ker}([l^{n}] : \mathrm{Br}(X) \rightarrow \mathrm{Br}(X)). \]
We call $\lambda(X) : = \mathrm{rank}_{\mathbb{Z}_{l}}(T_{l}(\mathrm{Br}(X)))$
the \textit{Lefschetz number} of $X$.
It is a birational invariant.
The Kummer sequence $0 \rightarrow \mu_{l^n} \rightarrow \mathbb{G}_{m} \overset{\times l^{n}}\rightarrow \mathbb{G}_{m} \rightarrow 0$
gives an exact sequence
\[ 0 \rightarrow \mathrm{NS}(X) \otimes \mathbb{Z}_{l} \rightarrow H^{2}_{\textit{\'et}}(X, \mathbb{Z}_{l}(1))
\rightarrow  T_{l}(\mathrm{Br}(X))  \rightarrow 0. \]
Thus, we have 
\[ \lambda(X) = b_{2}(X) - \rho(X), \]
where $b_{2}(X)$ and $\rho(X)$ denote the second Betti number and the Picard number of $X$, respectively.
Since $b_{2}$ is the independent of $l$, so also is $\lambda$. 
\begin{definition}
A surface $X$ is \textit{Shioda-supersingular} if $\lambda(X) = 0$ i.e., $b_{2}(X) = \rho(X)$.
\end{definition}
In particular, Conjecture \ref{Bloch conj} becomes the following statement$:$
\begin{conj} \label{Bloch conj 2} (= Conjecture \ref{Bloch conj}).
If $X$ is Shioda-supersingular surface, then 
\[ t_{2}(X) = 0.\]
\end{conj}
In this paper, we prove Conjecture \ref{Bloch conj 2} for quasi-elliptic surfaces (Theorem \ref{Intro theorem}).
Now, we recall the following property of Lefschetz numbers$:$
\begin{lem} (\cite[Lemma,\ p.234]{Shioda}). \label{innequality of Lefscetz number}
Let $\phi : Y \dashrightarrow X$ be a dominant rational map of surfaces over $k$.
Then 
\[\lambda(Y) \geq \lambda(X). \]
\end{lem}
For the reader's convenience, we include a proof of the following fact due to Shioda$:$
\begin{cor} (\cite[Corollary 2,\ p.235]{Shioda}).
\label{uniruled surface is Shioda-supersingular}
Any uniruled surface is Shioda-supersingular.
\end{cor}
\begin{proof}
Let $X$ be a uniruled surface.
By definition, there is a dominant rational map $\phi : \mathbb{P}^{1} \times C \dashrightarrow X$ 
for some curve $C$.
Now, one has
\[ b_{2}(\mathbb{P}^{1} \times C) = \rho(\mathbb{P}^{1} \times C) = 2. \]
\indent (Indeed, we have $H^{2}_{\textit{\'et}}(\mathbb{P}^{1} \times C, \mathbb{Q}_{l}) \cong \oplus_{i+j = 2}H^{i}_{\textit{\'et}}(\mathbb{P}^{1}, \mathbb{Q}_{l}) \otimes H^{j}_{\textit{\'et}}(C, \mathbb{Q}_{l})$
by the K\"unneth decomposition.
Then, we have $H^{2}_{\textit{\'et}}(\mathbb{P}^{1}, \mathbb{Q}_{l}) 
= H^{2}_{\textit{\'et}}(C, \mathbb{Q}_{l}) = \mathbb{Q}_{l}$ 
by Poincare duality. 
Since both $\mathbb{P}^{1}$ and $C$ are irreducible,
we have $H^{0}_{\textit{\'et}}(\mathbb{P}^{1}, \mathbb{Q}_{l}) = H^{0}_{\textit{\'et}}(C, \mathbb{Q}_{l}) = \mathbb{Q}_{l}$.
Thus, we get $b_{2}(\mathbb{P}^{1} \times C) = 2$ by $H^{1}_{\textit{\'et}}(\mathbb{P}^{1}, \mathbb{Q}_{l}) = 0$.
On the other hand, we have $\mathrm{NS}(\mathbb{P}^{1} \times C) = \mathrm{NS}(\mathbb{P}^{1}) \oplus \mathrm{NS}(C) 
\oplus \mathrm{Hom}(\mathrm{Jac}(\mathbb{P}^{1}), \mathrm{Jac}(C))$.
Since both $\mathbb{P}^{1}$ and $C$ has dimension $1$, we have 
$\mathrm{NS}(\mathbb{P}^{1}) = \mathrm{NS}(C) = \mathbb{Q}_{l}$.
Thus, we get $\rho(\mathbb{P}^{1} \times C) = 2$ by $\mathrm{Jac}(\mathbb{P}^{1}) = 0$.
Therefore, we get $b_{2}(\mathbb{P}^{1} \times C) = \rho(\mathbb{P}^{1} \times C) = 2$.)\\
\indent Namely, $\lambda(\mathbb{P}^{1} \times C) = 0$.
Since $\phi$ is dominant, $\lambda(X) = 0$ by Lemma \ref{innequality of Lefscetz number}.
Thus, $X$ is Shioda-supersingular.
\end{proof}

\subsection{Quasi-elliptic surfaces} \ \indent \\
In this subsection, we recall the uniruledness of quasi-elliptic surfaces
(Theorem \ref{coverling for quasi-elliptic}).\\
Let us begin with the following definition$:$
\begin{definition}
A \textit{genus $1$ fibration} from a surface is a proper morphism
\[ f : X \rightarrow C \]
from a smooth, relatively-minimal surface $X$ onto a normal curve $C$
such that the generic fiber $X_{\eta}$ is a normal, geometrically-integral, curve with arithmetic genus $1$.\\
\indent The fibration $f$ is called \textit{quasi-elliptic} (resp. \textit{elliptic}) 
if the geometric generic fiber $X_{\bar{\eta}}$ is \textit{not normal} (resp. \textit{normal}). 
\end{definition}

\begin{rem}
In fact, if $f$ is quasi-elliptic, then $X_{\bar{\eta}}$ is a singular rational curve with one cusp.
Quasi-elliptic surfaces can occur only in characteristic $2$ and $3$ (e.g. \cite{BombieriMumford3}).
\end{rem}

The following result plays a key role in the proof of Theorem \ref{Intro theorem}. 

\begin{thm} 
\label{coverling for quasi-elliptic}
Let $f : X \rightarrow C$ be a quasi-elliptic surface
over an algebraically closed field $k$ of characteristic $p > 0$.
Then, there are a birationally ruled surface $Y$ and a proper map $\pi : Y \rightarrow X$ of degree $p$.
More precisely, any quasi-elliptic surface is (purely inseparable) uniruled.
\end{thm}
\begin{proof}
The ideas of the proof are based on \cite[Section 1]{BombieriMumford3} or \cite[Theorem 9.4,\ p.266]{LiedtkeSurvey}.
Let $F : C^{(1/p)} \rightarrow C$ be the Frobenius morphism of degree $p$.
Let $K$ and $L$ be the functions fields of $C$ and $C^{(1/p)}$, respectively.
Let $X_{\eta}$ be the generic fiber of $f$.
Since $f$ is quasi-elliptic, $X_{\eta} \otimes_{K} L$ is not normal.
Let $Y_{\xi}$ be the normalization of $X_{\eta} \otimes_{K} L$.
Then $Y_{\xi}$ has arithmetic genus $0$.
Let $\phi : Y \rightarrow C^{(1/p)}$ be a regular, relatively minimal model of $Y_{\xi}$.
Then, there are the following commutative diagrams

\[
\begin{CD}
Y @>>> X \times_{C} C^{(1/p)} @>{\pi}>> X \\
@V{\phi}VV @VVV  @VV{f}V \\
C^{(1/p)} @= C^{(1/p)} @>>{F}> C
\end{CD}
\]

Now, the generic fiber of $\phi$ is $Y_{\xi}$, and $L = k(C^{(1/p)})$ is algebraically closed in $k(Y)$ 
(since $\phi_{*}\mathcal{O}_{Y} \cong \mathcal{O}_{C^{(1/p)}}$) .
By Noether-Tsen's theorem (Theorem \ref{Noether-Tsen}),
$Y$ is birationally isomorphic to $\mathbb{P}^{1} \times C^{(1/p)}$.
Hence, we get a dominant rational map
\[ \mathbb{P}^{1} \times C^{(1/p)} \dashrightarrow  X. \]
Since $L/K$ is purely inseparable, $X$ is purely inseparable uniruled.
\end{proof} 

\begin{rem}
Any genus $1$ fibration has Kodaira dimension $- \infty$, $0$, or $1$ 
(e.g. \cite{Badescu}).
By Proposition \ref{separably uniruled and birationally ruled} and Theorem \ref{coverling for quasi-elliptic},
any quasi-elliptic surface has Kodaira dimension $0$ or $1$.
\end{rem}

\begin{rem} By Theorem \ref{coverling for quasi-elliptic} and Corollary \ref{uniruled surface is Shioda-supersingular}, we have 
\begin{center}
quasi-elliptic \ \ \ $\Longrightarrow$ \ \ \ uniruled \ \ \ $\Longrightarrow$ \ \ \ Shioda-supersingular
\end{center}
\end{rem}

\section{Transcendental motives}
\subsection{Chow motives} \ \indent \\
In this subsection, we recall the notions of Chow motives and transcendental motives.\\
Let $k$ be an algebraically closed field of characteristic $p \geq 0$.
Let $\mathcal{V}(k)$ be the category of smooth projective varieties over $k$.
For every $V \in \mathcal{V}(k)$,
we denote by $\mathrm{CH}^{i}(V)$ the Chow group of codimensional $i$-cycles with $\mathbb{Q}$-coefficients,  and $\mathrm{CH}(V) = \oplus_{i}\mathrm{CH}^{i}(V)$.
Let $U, V, W \in \mathcal{V}(k)$.
For $\alpha \in \mathrm{CH}(U \times V)$, $\beta \in \mathrm{CH}(V \times W)$, define
$\beta \circ \alpha : = p_{UW*}(p_{UV}^{*}(\alpha) \cdot p_{VW}^{*}(\beta)) \in \mathrm{CH}(U \times W)$
where $p_{UV}$, $p_{VW}$, and $p_{UW}$ are the appropriate projections.
We denote by $\mathcal{M}_{rat}(k)$
the contravariant category of Chow motives with $\mathbb{Q}$-coefficients over $k$, which is $\mathbb{Q}$-linear, pseudoabelian, tensor category.
An object $M$ of $\mathcal{M}_{rat}(k)$ is the triple $M = (V, p, m)$,
where $V \in \mathcal{V}(k)$,
$p \in \mathrm{CH}^{\mathrm{dim}(V)}(V \times V)$ a projector (i.e., $p \circ p = p$),
and $m \in \mathbb{Z}$.
If $V, W \in \mathcal{V}(k)$ are irreducible, then 
\[ \mathrm{Hom}_{\mathcal{M}_{rat}(k)}((V, p, m), (W, q, n)) 
= q \circ \mathrm{CH}^{\mathrm{dim}(V) + n - m}(V \times W) \circ p. \]
\indent For $M = (V, p, m), N = (W, q, n) \in \mathcal{M}_{rat}(k)$,  
we denote by $M \oplus N$ the sum and by the tensor product  $M \otimes N$.
In particular, if $m = n$, then $M \oplus N  = (V \sqcup W, p \oplus q, m)$.
For a non-negative integer $n$, let 
$\mathbb{L}^{\oplus n} : = \mathbb{L} \oplus \cdot \cdot \cdot \oplus \mathbb{L}$ 
and $\mathbb{L}^{\otimes n} = \mathbb{L} \otimes \cdot \cdot \cdot \otimes \mathbb{L}$ ($n$-times).
For a prime number $l \neq p$, we consider the $l$-adic \'etale cohomology theory 
$H_{\textit{\'et}}^{*}$ which induces a functor
$H_{\textit{\'et}}^{*} : \mathcal{M}_{rat}(k) \rightarrow Vect_{\mathbb{Q}_{l}}^{gr}$ such that 
$H_{\textit{\'et}}^{i}((V, p, m)) = p^{*}H_{\textit{\'et}}^{i-2m}(V, \mathbb{Q}_{l})$.\\
\indent Let $V \in \mathcal{V}(k)$ be a variety of dimension $d$.
We denote by $h(V) = (V, \Delta_{V}, 0)$ the Chow motive of $V$.
Here $\Delta_{V}$ is the diagonal of $V$ in $\mathrm{CH}^{d}(V \times V)$.
If $d \leq 2$, then $V$ admits a Chow-K\"unneth decomposition, that is, there is a decomposition
\[ h(V) \cong \oplus_{i=0}^{2d}h_{i}(V) \]
such that $h_{i}(V) = (V, \pi_{i}(V), 0)$, $\pi$ are pairwise orthogonal projectors, and 
 $cl(\pi_{i})$ coincides with $(2d-i, i)$-component of $\Delta_{V}$ 
in the K\"unneth component of $H^{2d}_{\textit{\'et}}(V \times V, \mathbb{Q}_{l})$.
Here $cl : \mathrm{CH}^{d}(V \times V)_{hom} \rightarrow H^{2d}_{\textit{\'et}}(V \times V, \mathbb{Q}_{l})$ is the cycle map. Then $h_{0}(V) \cong 1$ and $h_{2d}(V) \cong \mathbb{L}^{\otimes 2d}$.
In particular, $h(\mathbb{P}) = 1 \oplus \mathbb{L}$.
Moreover, for two curves $C, D \in \mathcal{V}(k)$, by \cite[6.1.5,\ p.69]{Murre and  Nagel and Peters},
\begin{equation} \label{product of curves}
h(C \times D) \cong \oplus_{i=0}^{2}\oplus_{j+k = i}h_{j}(C) \otimes h_{k}(D).
\end{equation}
\indent From now on, let $X \in \mathcal{V}(k)$ be a surface. 
The motive $h_{1}(X)$ (resp. $h_{3}(X)$) is controlled by the Picard (resp. Albanese) variety of $X$. 
Thus, $h_{i}$ is well understood for $i \neq 2$.
Let 
\[ h_{2}(X) = h_{2}^{alg}(X) \oplus t_{2}(X) = (X, \pi_{2}^{alg}(X), 0) \oplus (X, \pi_{2}^{tr}(X), 0) \]
be the decomposition of $h_{2}(X)$ as in \cite{KMP}.
The motive $t_{2}(X)$ is called the transcendental motive of $X$.
It is a birational invariant and  
 $H^{2}_{\textit{\'et}}(h_{2}^{alg}(X)) = \mathrm{NS}(X)_{\mathbb{Q}_{l}}$ and 
$H^{2}_{\textit{\'et}}(t_{2}(X)) = H^{2}_{\textit{\'et}}(X, \mathbb{Q}_{l})_{tr}$.
By construction, $h_{2}^{alg}(X) \cong \mathbb{L}^{\oplus \rho(X)}$,
so $h_{2}^{alg}$ is also well understood.
However, $t_{2}$ is still mysterious.
For example, see Conjecture \ref{Bloch conj} (= Conjecture \ref{Bloch conj 2}).\\
\indent Now, we prove a necessary and sufficient condition for $t_{2} = 0$.
\begin{lem} \label{equivalence for t_{2}}
Let $X \in \mathcal{V}(k)$ be a surface. Then  
\[ h_{2}(X) \cong \mathbb{L}^{\oplus b_{2}(X)} \ \ \ \ \ \text{if and only if} \ \ \ \ \ t_{2}(X) = 0. \]
In particular, if $b_{2}(X) \neq \rho(X)$, then $t_{2}(X) \neq 0$.
\end{lem}
\begin{proof}
Assume $h_{2}(X) \cong \mathbb{L}^{\oplus b_{2}(X)}$.
Since $h_{2}^{alg}(X) \cong \mathbb{L}^{\oplus \rho(X)}$,
we have $t_{2}(X) \cong \mathbb{L}^{b_{2} - \rho}$.
Since $\mathrm{Hom}(\mathbb{L}, t_{2}(X)) = 0$ (e.g. \cite{KMP}), we have
$\mathrm{Hom}(\mathbb{L}^{b_{2}(X) - \rho(X)}, t_{2}(X)) = 0$, so $t_{2}(X) = 0$.\\
\indent Conversely, assume $t_{2}(X) = 0$.
Then $h_{2}(X) = h_{2}^{alg}(X) \cong \mathbb{L}^{\oplus \rho(X)}$.
Take the cohomology$:$ 
$H^{2}_{\textit{\'et}}(t_{2}(X)) = H^{2}_{\textit{\'et}}(X, \mathbb{Q}_{l})_{tr} \cong \mathbb{Q}_{l}^{b_{2}-\rho}$.
Since $t_{2}(X) = 0$, we have $\mathbb{Q}_{l}^{b_{2}-\rho} = 0$, so $b_{2} = \rho$.
Thus, $h_{2}(X) \cong \mathbb{L}^{\oplus b_{2}(X)}$.
On the contrary, if $b_{2} \neq \rho$, then $t_{2}(X) \neq 0$.
\end{proof}

\subsection{Homomorphisms between transcendental motives} \ \indent \\
In this subsection, we prove some results on homomorphisms between transcendental motives.
Let $k$ be an algebraically closed field. Let $X, Y \in \mathcal{V}(k)$ be surfaces.
\begin{center}
$\mathrm{CH}_{2}(X \times Y)_{\equiv}$ $:$ 
the subgroup of $\mathrm{CH}_{2}(X \times Y)$ generated by the classes supported on subvarieties of the form 
$X \times N$ or $M \times Y$, 
with  $M$ a closed subvariety of $X$ of dimension $< 2$ 
and $N$ a closed subvariety of $Y$ of dimension $< 2$.
\end{center}
We define a homomorphism 
\begin{align*}
\Phi_{X, Y} : \mathrm{CH}_{2}(X \times Y) &\rightarrow 
\mathrm{Hom}_{\mathcal{M}_{rat}(k)}(t_{2}(X), t_{2}(Y)) \\
  \alpha &\mapsto \pi_{2}^{tr}(Y) \circ \alpha \circ \pi_{2}^{tr}(X).
\end{align*}
\begin{thm} \label{automorophism groups of motives}
(\cite[Theorem\ 7.4.3,\ p.165]{KMP}). There is an isomorphism of groups
\[ \mathrm{CH}_{2}(X \times Y)/\mathrm{CH}_{2}(X \times Y)_{\equiv} \cong 
\mathrm{Hom}_{\mathcal{M}_{rat}(k)}(t_{2}(X), t_{2}(Y)). \]
\end{thm}

To prove the functorial relation for $\Phi_{X, Y}$, we need the following lemma$:$
\begin{lem} \label{composition of mod} 
Let $\alpha \in \mathrm{CH}_{2}(X \times Y)$ and 
$\gamma \in \mathrm{CH}_{2}(Y \times X)_{\equiv}$. Then
\[ \text{(i) $\gamma \circ \alpha  \in \mathrm{CH}_{2}(X \times X)_{\equiv}$ \ \ \ \ \ 
and \ \ \ \ \ (ii) $\alpha \circ \gamma \in \mathrm{CH}_{2}(Y \times Y)_{\equiv}$}. \]
\end{lem}
\begin{proof}
The proof of (ii) is similar to (i).
Thus, it suffices to prove (i).
Without loss of generality, we may assume that $\gamma$ is irreducible and supported
on $Y \times C$ with $\mathrm{dim}(C) \leq 1$.\\
\indent First, assume $\mathrm{dim}(C) = 0$.
Let $p \in X$ be a closed point.
For $\gamma = [Y \times p]$, then
\[\gamma \circ \alpha = [Y \times p] \circ \alpha 
= p_{YY*}^{YXY}(\alpha \times Y \cdot Y \times X \times p) 
= p_{YY*}^{YXY}(\alpha \times p) = [p_{Y*}^{YX}(\alpha) \times p]. \]
Thus $\gamma \circ \alpha \in \mathrm{CH}_{2}(X \times X)_{\equiv}$.
Next, assume $\mathrm{dim}(C) = 1$.
Since $\gamma$ is supported on $Y \times C$, there are a smooth irreducible curve $C$ and 
a closed embedding $\iota : C \hookrightarrow X$ such that
$\gamma = \Gamma_{\iota} \circ D \ \text{in} \ \mathrm{CH}_{2}(Y \times X)$, 
where $\Gamma_{\iota} \in \mathrm{CH}_{1}(C \times X)$ is the graph of $\iota$
and $D \in \mathrm{CH}_{2}(Y \times C)$.
Since the support of the second projection of $\Gamma_{\iota}$ has dimension $\leq 1$,
the support of the second projection of $\gamma \circ \alpha$ has dimension $\leq 1$, 
and hence $\gamma \circ \alpha \in \mathrm{CH}_{2}(X \times X)_{\equiv}$.\\
\end{proof}
The following result is the functorial relation for $\Phi_{X, Y}$$:$
\begin{prop} (\cite[p.62]{Ped}). \label{functoriality of transcendental motives}
For surfaces $X$, $Y$, $Z \in \mathcal{V}(k)$, 
\[ \Psi_{Y, Z}(\beta) \circ \Psi_{X, Y}(\alpha) = \Psi_{X, Z}(\beta \circ \alpha)
\ \ \text{in} \ \ \mathrm{Hom}_{\mathcal{M}_{rat}(k)}(t_{2}(X), t_{2}(Z)). \]
\end{prop}
\begin{proof}
Let $\Delta_{Y} = \pi_{0} + \pi_{1} + \pi_{2}^{alg} + \pi_{2}^{tr} + \pi_{3} + \pi_{4}$
be the CK-decomposition in $\mathrm{CH}_{2}(Y \times Y)$.
Since $\pi_{2}^{tr}(Y) \circ \pi_{2}^{tr}(Y) = \pi_{2}^{tr}(Y)$, it suffices to prove
in $\mathrm{Hom}(t_{2}(X), t_{2}(Z))$
\[ \pi_{2}^{tr}(Z) \circ \beta \circ \pi_{2}^{tr}(Y) \circ \alpha \circ \pi_{2}^{tr}(X) 
= \pi_{2}^{tr}(Z) \circ \beta \circ \alpha \circ \pi_{2}^{tr}(X).\] 
By Theorem \ref{automorophism groups of motives}, it suffices to prove 
\[ \beta \circ \pi_{2}^{tr}(Y) \circ \alpha -  \beta \circ \alpha \in \mathrm{CH}_{2}(X \times Z)_{\equiv}. \]
By the constructions of $\pi_{i}$ for $i \neq 2$ and $\pi_{2}^{alg}$ (e.g. \cite{KMP}),
\[ \pi_{i}(Y) \in \mathrm{CH}_{2}(Y \times Y)_{\equiv} \ \ \ \text{and} \ \ \ 
\pi_{2}^{alg}(Y) \in \mathrm{CH}_{2}(Y \times Y)_{\equiv}. \]
By Lemma \ref{composition of mod}, 
\begin{equation} \label{equality for t2}
\beta \circ \pi_{i}(Y) \circ \alpha \in \mathrm{CH}_{2}(X \times Z)_{\equiv} \ \ \ \text{and} \ \ \ 
\beta \circ \pi_{2}^{alg}(Y) \circ \alpha \in \mathrm{CH}_{2}(X \times Z)_{\equiv} 
\end{equation}
\indent Therefore, we get
\begin{align*}
\beta \circ \pi_{2}^{tr}(Y) \circ \alpha - \beta \circ \alpha 
&= \beta \circ (\Delta_{Y} - \pi_{0}(Y) - \pi_{4}(Y) - \pi_{2}^{alg}(Y)- \pi_{1}(Y) - \pi_{3}(Y)) \circ \alpha 
- \beta \circ \alpha \\
&\overset{(\ref{equality for t2})}= \alpha \circ (- \pi_{0}(Y) - \pi_{4}(Y) - \pi_{2}^{alg}(Y)- \pi_{1}(Y) - \pi_{3}(Y)) \circ  \beta  
\ \ \text{in} \ \ \mathrm{CH}_{2}(X \times Z)_{\equiv} 
\end{align*}
\end{proof}

Using Proposition \ref{functoriality of transcendental motives}, we prove the following$:$
\begin{lem} \label{finite morphism for t_{2}}
Let $\pi : Y \rightarrow X$ be a finite morphism of surfaces.
Let $\Gamma_{\pi} \in \mathrm{CH}^{2}(Y \times X)$ be the graph of $\pi$
and $^{t}\Gamma_{\pi}$ its transpose.
Then there is an isomorphism of Chow motives 
\[ t_{2}(Y) \cong t_{2}(X) \oplus (Y, \pi_{2}^{tr}(Y) -  \Psi_{X, Y}(\Gamma) \circ \Psi_{Y, X}(^{t}\Gamma), 0). \]
\end{lem}
\begin{proof} Let $d$ be the degree of $\pi$.
We let $p : = 1/d \cdot \Psi_{X, Y}(^{t}\Gamma_{\pi}) \circ \Psi_{Y, X}(\Gamma_{\pi})$.\\
\indent (i) We prove that $p$ and $\pi_{2}^{tr}(Y) - p$ are pairwise orthogonal projectors.
In $\mathrm{Hom}(t_{2}(Y), t_{2}(Y))$,
\begin{align*}
p \circ p &= 1/d^{2} \cdot \Psi_{X, Y}(^{t}\Gamma_{\pi}) \circ \Psi_{Y, X}(\Gamma_{\pi}) \circ 
\Psi_{X, Y}(^{t}\Gamma) \circ \Psi_{Y, X}(\Gamma)\\
&=  1/d^{2} \cdot \Psi_{X, Y}(^{t}\Gamma_{\pi} \circ \Gamma_{\pi} \circ ^{t}\Gamma_{\pi} \circ \Gamma_{\pi})
&& \text{by Proposition \ref{functoriality of transcendental motives}} \\
&= 1/d \cdot \Psi_{X, Y}(^{t}\Gamma_{\pi} \circ \Gamma_{\pi}) 
&& \text{by $\Gamma_{\pi} \circ ^{t}\!\Gamma_{\pi} = d \cdot \Delta_{X}$} \\
&= 1/d \cdot \Psi_{X, Y}(^{t}\Gamma_{\pi}) \circ \Psi_{Y, X}(\Gamma_{\pi}) 
&& \text{by Proposition \ref{functoriality of transcendental motives}} \\
&= p.
\end{align*} 
Thus $p$ is a projector.
Similarly, one has $p \circ \pi_{2}^{tr}(Y) = \pi_{2}^{tr}(Y) \circ p  = p$.
Thus, $\pi_{2}^{tr}(Y) - p$ is also projector, and $p$ and $\pi_{2}^{tr}(Y) - p$ are orthogonal.\\
\indent (ii) We prove $t_{2}(X) \cong (Y, p, 0)$. We let 
\begin{align*}
\alpha &: =  1/d \cdot p \circ \Phi_{X, Y}(^{t}\Gamma_{\pi}) \circ \pi_{2}^{tr}(X) \in \mathrm{Hom}(t_{2}(X), (Y, p, 0))\\
\beta &: = 1/d \cdot \pi_{2}^{tr}(X) \circ \Psi_{Y, X}(\Gamma_{\pi}) \circ p \in \mathrm{Hom}((Y, p, 0), t_{2}(X)).
\end{align*}
By the same way as in (i),
we have $\alpha \circ \beta = p$ and $\beta \circ \alpha = \pi_{2}^{tr}(X)$, so we get $t_{2}(X) \cong (Y, p, 0)$.\\
\indent (iii) We prove $t_{2}(Y) \cong t_{2}(X) \oplus (Y, \pi_{2}^{tr}(Y) - p, 0)$.
By (i) and (ii), we get isomorphisms
\[ t_{2}(Y) \overset{(\mathrm{i})}\cong (Y, p, 0) \oplus (Y, \pi_{2}^{tr}(Y) - p, 0) \overset{(\mathrm{ii})}\cong
t_{2}(X) \oplus (Y, \pi_{2}^{tr}(Y) - p, 0).\]
Thus, we completes the proof of Lemma \ref{finite morphism for t_{2}}.
\end{proof}

To prove $t_{2} = 0$ for uniruled surfaces (Theorem \ref{Main 1}), we need the following$:$

\begin{lem} (\cite[p.66]{Ped}). \label{dominant morphism for t_{2}}
Let $\phi : Y \dashrightarrow X$ be a dominant rational map of surfaces.
Then $t_{2}(X)$ is the direct summand of $t_{2}(Y)$, that is,
there are a motive $M$ and a decompostion 
\[ t_{2}(Y) \cong t_{2}(X) \oplus M. \]
\vspace*{-1.5\baselineskip}
\begin{proof}
By the elimination of indeterminacy of $\phi$ (since $\mathrm{dim}(X) =2$), 
there are a surface $Z$,
a birational morphism $\psi : Z \rightarrow Y$,
and a finite surjective morphism $\pi : Z \rightarrow X$ such that the diagram
$$\xymatrix{
& Z \ar[dr]^\pi \ar[dl]_\psi & \\
Y \ar@{.>}[rr]^{\phi} & &  \ X 
}$$
is commutative.
By Lemma \ref{finite morphism for t_{2}},
there is a decomposition $t_{2}(Z) \cong t_{2}(X) \oplus M$ for some motive $M$.
Since $t_{2}$ is a birational invariant, $t_{2}(Y) \cong t_{2}(Z)$.
Therefore, we get $t_{2}(Y) \cong t_{2}(X) \oplus M$ for some motive $M$.
\end{proof}
\end{lem}

\section{Proof of Main theorem}

To prove our main theorem, we prove the following$:$

\begin{thm}\label{Main 1}
Let $X$ be a uniruled surface. Then $t_{2}(X) = 0$.
\begin{proof}
Since $X$ is uniruled, there are a curve $C$ and a dominant rational map 
\[ \phi : \mathbb{P}^{1} \times C \dashrightarrow X. \]
By Lemma \ref{dominant morphism for t_{2}},
there are a motive $M$ and a decomposition 
\[ t_{2}(\mathbb{P}^{1} \times C ) \cong t_{2}(X) \oplus M. \]
Thus, it suffices to prove $t_{2}(\mathbb{P}^{1} \times C) = 0$.
Indeed, there is the CK-decomposition
\[ h_{2}(\mathbb{P}^{1} \times C) \cong \oplus_{j + k = 2}h_{j}(\mathbb{P}^{1}) \otimes h_{k}(C) \]
by (\ref{product of curves}).
Since both $\mathbb{P}^{1}$ and $C$ has dimension $1$,
we have $h_{0}(-) = 1$ and $h_{2}(-) = \mathbb{L}$, so we get $h_{2}(\mathbb{P}^{1} \times C) \cong  \mathbb{L}^{\oplus 2}$ because $h_{1}(\mathbb{P}) = 0$.
By the argument as in the proof of Corollary \ref{uniruled surface is Shioda-supersingular},
we have $b_{2}(\mathbb{P}^{1} \times C) = \rho(\mathbb{P}^{1} \times C) = 2$.
Thus, we have $h_{2}(\mathbb{P}^{1} \times C) \cong \mathbb{L}^{\oplus b_{2}(\mathbb{P}^{1} \times C)}$.
By Lemma \ref{equivalence for t_{2}}, we get $t_{2}(\mathbb{P}^{1} \times C) = 0$.
This completes the proof of Theorem \ref{Main 1}.
\end{proof}
\end{thm}

\begin{rem}
Let $C$ and $D$ be smooth projective curves over $\mathbb{C}$ with positive genus.
Let $X = C \times D$.
Then $p_{g}(X) = p_{g}(C) \cdot p_{g}(D) > 0$.
By \cite[pp.155-156]{ShiodaSupersingular}, $b_{2}(X) \neq \rho(X)$.
By Lemma \ref{equivalence for t_{2}}, 
we get $t_{2}(X) \neq 0$, that is, $h_{2}(X) \neq \mathbb{L}^{\oplus b_{2}(X)}$.
\end{rem}

Our main theorem is the following$:$

\begin{thm} (= Theorem \ref{Intro theorem}). \label{Main 2}
Let $f : X \rightarrow C$ be a quasi-elliptic surface. Then 
\[ t_{2}(X) = 0.\]
\end{thm}
\begin{proof}
By Theorem \ref{coverling for quasi-elliptic}, $X$ is uniruled.
By Theorem \ref{Main 1}, $t_{2}(X) = 0$.
\end{proof}

\subsection*{Acknowledgements}
I would like to express my special appreciation and thanks to 
Prof.~Hanamura for helpful comments and suggestions.
I am deeply grateful to Prof.~Tsuzuki for his generous support and comments.
I would like to thank Prof.~Katsura for helpful discussions and comments,  
Prof. ~Ogawa for warm encouragement, and the referee for his/her specific comments.
The author is supported by the JSPS KAKENHI Grant Number 18H03667.

\end{document}